\documentclass[12pt]{amsart}
\RequirePackage[sc]{mathpazo}
\counterwithin{equation}{section}
\usepackage{eucal}

\usepackage[centering,scale=0.75]{geometry}

\usepackage[all]{xy}
\usepackage{amssymb}
\usepackage{mathrsfs}
\usepackage{graphicx}
\usepackage{xcolor}
\usepackage{tikz}
\usepackage[colorlinks,plainpages,backref,
    linkcolor=blue!80!green,
    citecolor=green!60!blue,
    urlcolor=red!50!black]{hyperref}

\newtheorem{theorem}{Theorem}[section]

\newtheorem{lemma}[theorem]{Lemma}

\newtheorem{conj}[theorem]{Conjecture}

\newtheorem{remark}[theorem]{Remark}

\newcommand{\Htilde}{\widetilde H}

\allowdisplaybreaks

\begin{document}

\title{A solution to Butler's positivity  conjecture}

\author{Peter L. Guo}
\address[Peter L. Guo]{Center for Combinatorics, LPMC, 
Nankai University, Tianjin 300071, P.R. China}
\email{lguo@nankai.edu.cn}

\author{Mingyang Kang}
\address[Mingyang Kang]{Center for Combinatorics, LPMC, 
Nankai University, Tianjin 300071, P.R. China}
\email{1120240006@mail.nankai.edu.cn}

\author{Rui Xiong}
\address[Rui Xiong]{Department of Mathematics and Statistics, University of Ottawa, 150 Louis-Pasteur, Ottawa, ON, K1N 6N5, Canada}
\email{rxion043@uottawa.ca}

\maketitle

% \vspace{-20}

\begin{abstract}
Let $\lambda, \mu, \nu$   be  distinct partitions
such that $\lambda, \mu\subset \nu$ and $|\nu/\lambda|=|\nu/\mu|=1$. 
We prove Butler's positivity  conjecture posed in 1994:  the expansion of the  Macdonald intersection polynomial
\[
\frac{T_\lambda\widetilde{H}_\mu(X;q,t)-T_\mu\widetilde{H}_\lambda(X;q,t)}
     {T_\lambda-T_\mu}
\]
in terms of the Schur function basis has  coefficients in $\mathbb{Z}_{\geq 0}[q,t]$. 
\end{abstract}

\section{Introduction}

Let  $\nu=(\nu_1,\nu_2,\ldots)\vdash n$ be a partition of  $n$, namely, $\nu_1\geq \nu_2\geq \cdots$ and $|\nu|=\nu_1+\nu_2+\cdots=n$. Write  $\Htilde_\nu(X;q,t)$ for the {\it modified Macdonald polynomial}. 
Originally conjectured by 
Macdonald \cite{Mac-2} and ultimately resolved  by Haiman \cite{Haiman2001},      $\Htilde_\nu(X;q,t)$ for $\nu\vdash n$  is  Schur positive: 
\[
\Htilde_\nu(X;q,t)\in
\bigoplus_{\theta\vdash n}\mathbb{Z}_{\geq 0}[q,t]s_\theta(X),
\] 
where $s_\theta(X)$ is the Schur function indexed by $\theta$. 
See Haglund,     Haiman and  Loehr \cite{HML} for a combinatorial treatment of various properties concerning  $\Htilde_\nu(X;q,t)$.

Denote 
\[
D(\nu)=\{(r,s)\in\mathbb{Z}_{\geq 0}^2:0\le s<\nu_{r+1}\},
\]
which is the  diagram array  of $\nu$ using zero-based row and column coordinates.  For example, for $\nu=(3,2)$, we have  
\[
D(\nu)=\{(0,0), (0,1), (0,2), (1,0), (1,1)\}.
\]
 The conjugate  $\nu'$ of $\nu$ is the partition for which $D(\nu')=\{(s,r)\colon (r,s)\in D(\nu)\}$. For $\nu=(3,2)$, it is clear that $\nu'=(2,2,1)$.  For two partitions $\lambda$ and $\nu$, write  $\lambda\subseteq \nu$ if  $\lambda_i\leq \nu_i$ for $i\geq 1$. For $\lambda\subseteq \nu$, denote    $|\nu/\lambda|=|\nu|-|\lambda|$. 
 
Set
\begin{equation}\label{eq:Tweight}
T_\nu:=\prod_{(r,s)\in D(\nu)}t^rq^s=t^{n(\nu)}q^{n(\nu')},
\end{equation}
where 
\[n(\nu)=\sum_{i\geq 1} (i-1)\nu_i.\]

In 1994, Butler \cite{Butler1994} posed  the following Schur positivity conjecture.

\begin{conj}[\text{\cite{Butler1994}}]\label{conj-m}
Let $\nu$ be a partition of $n+1$, and $\lambda, \mu\subset \nu$  be two distinct partitions of $n$
such that $|\nu/\lambda|=|\nu/\mu|=1$. Then the  Macdonald intersection polynomial is Schur positive:
\[
I_{\lambda,\mu}(X;q,t)
:=\frac{T_\lambda\Htilde_\mu(X;q,t)-T_\mu\Htilde_\lambda(X;q,t)}
     {T_\lambda-T_\mu}\in
\bigoplus_{\theta\vdash n}\mathbb{Z}_{\geq 0}[q,t]s_\theta(X).
\]   
\end{conj}
 
\begin{remark}
(1)  Significant   progress  toward  Conjecture \ref{conj-m}  has been made by   Kim, Lee and Oh \cite[Theorem 1.3]{KLO2026} where    a weaker version  was proved: $I_{\lambda,\mu}(X;q,t)$ is monomial positive. In addition, they  confirmed   Conjecture \ref{conj-m} in some special cases \cite[Theorem 1.5]{KLO2026}. 

(2) It should be pointed out that Conjecture \ref{conj-m} is  implied by the elusive   Science Fiction Conjecture of Bergeron and Garsia  \cite{BG1999} concerning  the intersection of Garsia--Haiman  modules. See  \cite{BG1999,KLO2025,KLO2026} for more discussions. 
\end{remark}

The goal of this note is to provide  a proof of Conjecture \ref{conj-m} by realizing $I_{\lambda,\mu}(X;q,t)$ as the  Frobenius characteristic of a   bigraded  $S_n$-module, where  $S_n$ is the symmetric group  on $\{1,2,\ldots, n\}$.   Every finite-dimensional
complex $S_n$-module can be expressed as a direct sum of irreducible $S_n$-modules $S^\theta$   indexed by partitions  $\theta\vdash n$.  The  Frobenius characteristic map $\mathrm{ch}$ sends each irreducible module $S^\theta$ to $\mathrm{ch}(S^\theta)=s_\theta(X)$.   For a bigraded $S_n$-module (supported in bidegree  $\mathbb{Z}_{\geq 0}^2$)
\[
M= 
\bigoplus_{(r,s)\in\mathbb{Z}_{\geq 0}^2}M_{r,s},
\]
its bigraded  Frobenius characteristic is
\begin{equation*} 
\mathrm{Frob}_{q,t}(M)
=
\sum_{r,s\in\mathbb{Z}_{\geq 0}}t^rq^s\operatorname{ch}(M_{r,s}).
\end{equation*}

We prove Conjecture \ref{conj-m} by establishing the following.

\begin{theorem}\label{thm:main}
Let $\lambda$, $\mu$ and $\nu$ be as defined  in Conjecture \ref{conj-m}.   Then there exists a  bigraded $S_n$-module whose 
Frobenius characteristic is $I_{\lambda,\mu}(X;q,t)$. 
\end{theorem}

The proof of Theorem \ref{thm:main} relies on the geometry of   the Hilbert scheme of  points in the affine plane.

\subsection*{Acknowledgements}
The first author was  supported by the National Natural Science Foundation of China (No. 12371329) and the Fundamental Research Funds for the Central Universities (No. 63263094). We used ChatGPT 5.6 during  the preparation of this work.

\section{Hilbert scheme and proof of Theorem \ref{thm:main}}
Let $H_n=\operatorname{Hilb}^n(\mathbb{C}^2)$ be the Hilbert scheme of $n$ points in $\mathbb{C}^2$. That is, $H_n$ parametrizes the ideals of $\mathbb{C}[x,y]$ of codimension $n$:
$$H_n=\{I\trianglelefteq \mathbb{C}[x,y]\colon  
\dim_{\mathbb{C}} \mathbb{C}[x,y]/I=n\}.$$
By Fogarty \cite{Fogarty}, the Hilbert scheme $H_n$ is a non-singular variety of dimension $2n$. 
There is a \emph{Hilbert--Chow map}
$$\sigma\colon  H_n\to \mathbb{C}^{2n}/S_n.$$
Explicitly, we can identify $\mathbb{C}^{2n}/S_n$ with the variety of multi-sets of $\mathbb{C}^2$ of cardinality $n$, and the Hilbert--Chow map sends $I$ to the support (with multiplicity) of $\mathbb{C}[x,y]/I$ .
The isospectral Hilbert scheme $X_n$ is defined to be the reduced fiber product in the following diagram
\begin{equation}\label{eq:diagofHilb}
\begin{matrix}
\xymatrix{
X_n\ar[r] \ar[d]_\rho & \mathbb{C}^{2n}\ar[d]\\
H_n \ar[r]^-{\sigma} & \mathbb{C}^{2n}/S_n.}
\end{matrix}
\end{equation}
Induced from the natural action on $\mathbb{C}^2$, the torus $T=\mathbb{C}^\times \times \mathbb{C}^\times$ acts on $H_n$ and $X_n$. 
Moreover $X_n$ is equipped with an $S_n$-action induced from that of $\mathbb{C}^{2n}$, and the map $\rho$ is $S_n$-equivariant if we view $H_n$ as an $S_n$-variety with trivial action.

We have two important bundles over $H_n$. 
\begin{itemize}
    \item 
The first bundle is the \emph{tautological bundle}
$\mathcal{B}$, whose fiber at an ideal $I$ is the algebra $\mathbb{C}[x,y]/I$. 
It is clear that $\mathcal{B}$ is a $T$-equivariant vector bundle of rank $n$. 

    \item 
The second bundle is the \emph{Procesi bundle} 
$\mathcal{P}=\rho_*\mathcal{O}_{X_n}$. 
By Haiman's theory on the $n!$ conjecture 
\cite[Conjecture 2.2.1]{Haiman2001}, 
$\mathcal{P}$ is a ($T\times S_n$)-equivariant vector bundle of rank $n!$; 
see also \cite[Prop 5.4.1]{HaimanSurvey}. 
\end{itemize}
Following Haiman \cite[Equation (102)]{Haiman2002}, we define a $T$-equivariant bundle 
$$\mathcal{O}_{H_n}(1) := \det \mathcal{B} \cong 
\operatorname{Hom}_{S_n}(\mathbf{sgn},\mathcal{P}),
$$
where $\mathbf{sgn}$ is the sign representation of $S_n$. 
Write $\mathcal{O}_{H_n}(k)$ for $\mathcal{O}_{H_n}(1)^{\otimes k}$. 
For a partition $\mu$, we define a monomial   ideal
$$I_\mu=\left<x^ay^b\colon (a,b)\notin D(\mu)\right>\trianglelefteq \mathbb{C}[x,y]. $$
Then the set of torus fixed points of $H_n$ is  
$H_n^T = \{I_\mu\colon \mu\vdash n\}$. 

In the remainder we let $\lambda, \mu, \nu$   be  distinct partitions
such that $\lambda, \mu\subset \nu$ and $|\nu/\lambda|=|\nu/\mu|=1$. 
We can construct a flat $\mathbb{P}^1$-family of ideals as follows. 
For $\lambda\subset \nu$, denote $D(\nu/\lambda)=D(\nu)\setminus D(\lambda)$.
Let $L$ be a one-dimensional subspace of $$\mathbb{C}^2\cong \operatorname{span}\left(x^ay^b\colon (a,b)=D(\nu/\lambda)\text{ or }D(\nu/\mu)\right).$$ 
We define $I_L=I_{\nu}+L$. 
Assume $|\lambda|=|\mu|=|\nu|-1=n$. 
This defines a $T$-equivariant morphism 
$$f:\mathbb{P}^1\longrightarrow H_n,\qquad L\longmapsto I_L.$$

\begin{lemma}\label{lem:O1isO1}
As a non-equivariant line bundle, we have  $f^*\mathcal{O}_{H_n}(1)\cong \mathcal{O}_{\mathbb{P}^1}(1)$. 
\end{lemma}
\begin{proof}
By construction, we have the following short exact sequence 
$$0\longrightarrow \mathcal{O}_{\mathbb{P}^1}(-1)
\longrightarrow 
\mathcal{O}_{\mathbb{P}^1}
\otimes \mathbb{C}[x,y]/I_\nu
\longrightarrow 
f^*\mathcal{B}
\longrightarrow 0. $$
Its fiber at $L$ is the short exact sequence 
$$0\longrightarrow L\longrightarrow \mathbb{C}[x,y]/I_\nu\longrightarrow 
\mathbb{C}[x,y]/I_L\longrightarrow 0.
$$
The middle sheaf is a trivial bundle. Therefore, as a non-equivariant bundle, we have
$f^*\mathcal{O}_{H_n}(1)=f^*\det\mathcal{B}=\det f^*\mathcal{B}=\mathcal{O}_{\mathbb{P}^1}(1)$. 
\end{proof}

\begin{lemma}\label{lem:Pisgg}
The sheaf $\mathcal{P}$ is generated by global sections. 
\end{lemma}
\begin{proof}
Since $X_n$ is constructed as a reduced fiber product, we have a closed embedding $i\colon X_n\to H_n\times \mathbb{C}^{2n}$. 
It gives a surjection
$\mathcal{O}_{H_n\times \mathbb{C}^{2n}}\twoheadrightarrow i_*\mathcal{O}_{X_n}$.
Since the first projection $\operatorname{pr}_1\colon H_n\times \mathbb{C}^{2n}\to H_n$ is affine, 
we get a surjection 
$$\mathcal{O}_{H_n}\otimes \mathbb{C}[\mathbb{C}^{2n}]
\cong 
\operatorname{pr}_{1,*}\mathcal{O}_{H_n\times \mathbb{C}^{2n}}\twoheadrightarrow \operatorname{pr}_{1,*}
i_*\mathcal{O}_{X_n}
=\rho_*\mathcal{O}_{X_n}=\mathcal{P}.
$$
Thus $\mathcal P$ is generated by global sections.
% Although $\mathbb{C}[\mathbb{C}^{2n}]=\mathbb{C}[x_1,\ldots,x_n,y_1,\ldots,y_n]$ is infinite dimensional, this does not weaken global generation. 
% Since $\mathcal{P}$ is coherent and $H_n$ is quasi-compact, it is impossible to have an infinite increasing chain in $\mathcal{P}$. 
% So there exists a finite dimensional subspace $V$ of 
% $\mathbb{C}[\mathbb{C}^{2n}]$
% % $\mathbb{C}[x_1,\ldots,x_n,y_1,\ldots,y_n]$ 
% such that the restriction 
% $$\mathcal{O}_{H_n}\otimes V\twoheadrightarrow \mathcal{P}$$
% is surjective. 
\end{proof}

\begin{theorem}
There exist $(T\times S_n)$-representations $M,N$ such that 
we have a short exact sequence of ($T\times S_n$)-bundle
\begin{equation}\label{eq:BGthmses}
0\longrightarrow 
f^*\mathcal{O}_{H_n}(1)\otimes N
\longrightarrow 
f^*\mathcal{P}\longrightarrow
\mathcal{O}_{\mathbb{P}^1}\otimes M
\longrightarrow 
0.
\end{equation}
\end{theorem}
\begin{proof}
Since $f^*\mathcal{P}$ is a vector bundle over $\mathbb{P}^1$, 
it is (non-equivariantly) isomorphic to a direct sum of line bundles by the Birkhoff--Grothendieck theorem \cite{Grothendieck1957}.  
By Lemma \ref{lem:Pisgg}, the line bundles can only be  chosen from $\mathcal{O}_{\mathbb{P}^1}(k)$ for $k\geq 0$. 
Haiman constructed a ($T\times S_n$)-equivariant isomorphism \cite[Equation (45)]{Haiman2001}\footnote{%
In loc. cit. the factor $\mathbf{sgn}$ was not included since the $S_n$-action is not considered. But an  explicit computation shows that $\mathbf{sgn}$ is needed to make the isomorphism $S_n$-equivariant; see the equation after \cite[Equation (96)]{HaimanSurvey}.}% 
:
\begin{equation}\label{eq:dualityP}
\tilde{\phi}\colon \mathcal{P}\longrightarrow 
\mathcal{P}^\vee\otimes \mathcal{O}_{H_n}(1)\otimes \mathbf{sgn}. 
\end{equation}
By Lemma \ref{lem:O1isO1}, the line bundle 
can only be chosen from $\mathcal{O}_{\mathbb{P}^1}(k)$ for $-k+1\geq 0$, i.e. $k\leq 1$. 

As a result, as a non-equivariant bundle $f^*\mathcal{P}$ is isomorphic to a direct sum of $\mathcal{O}_{\mathbb{P}^1}(k)$ for $k=0,1$. 
By the Harder--Narasimhan filtration, 
the direct sum of all line bundles for $k=1$ is canonical, in particular it is ($T\times S_n$)-equivariant. Explicitly, it is the image of the ($T\times S_n$)-equivariant morphism 
$$0\longrightarrow f^*\mathcal{O}_{H_n}(1)\otimes N
\longrightarrow 
f^*\mathcal{P},$$
where $N=H^0(\mathbb{P}^1,f^*\mathcal{P}\otimes f^*\mathcal{O}_{H_n}(-1))$. 
The cokernel $\mathcal{Q}$ is isomorphic to a trivial bundle over $\mathbb{P}^1$, so it is isomorphic to $\mathcal{O}_{\mathbb{P}^1}\otimes M$ where  $M=H^0(\mathbb{P}^1,\mathcal{Q})$. 
\end{proof}

Taking the dual of \eqref{eq:BGthmses}, we get the following short exact sequence
$$0\longrightarrow 
\mathcal{O}_{\mathbb{P}^1}\otimes M^*
\longrightarrow 
f^*\mathcal{P}^\vee \longrightarrow
f^*\mathcal{O}_{H_n}(-1)\otimes N^*
\longrightarrow 0.$$
Comparing   with \eqref{eq:BGthmses} and along with  \eqref{eq:dualityP}, we have 
$$M\cong N^*\otimes \mathbf{sgn}
=H^0(\mathbb{P}^1,f^*\mathcal{P}^\vee\otimes \mathbf{sgn})^*\otimes \mathbf{sgn}
= H^0(\mathbb{P}^1,f^*\mathcal{P}^\vee)^*.$$
In particular, $\dim M=\dim N=\frac{n!}{2}$. 

Let $D_\mu$ (resp., $R_\mu$) be the fiber of $\mathcal{O}_{H_n}(1)$ (resp., $\mathbb{P}$) at $I_\mu$. 
\begin{itemize}
    \item $D_\mu$ is a one-dimensional vector space with weight $T_\mu=t^{n(\mu)}q^{n(\mu')}$. 
    \item $R_\mu$ is isomorphic to the 
    Garsia-Haiman module. 
\end{itemize}

By Haiman's theory on the Macdonald positivity conjecture \cite[Conjecture 2.2.2, Equation (56)]{Haiman2001}, the bigraded Frobenius characteristic of $R_\mu$ is the modified  Macdonald polynomial $\widetilde{H}_\mu(X;q,t)$. 
For example, \eqref{eq:dualityP} implies the following  reciprocity identity  (see \cite[Proposition 3.5.12]{HaimanSurvey}):
$$\widetilde{H}_\mu(X;q,t) = t^{n(\mu)}q^{n(\mu')} \omega \widetilde{H}_\mu(X;q^{-1},t^{-1}),
$$
where $\omega$ is the omega  involution on symmetric functions which is defined by  sending  $s_\theta(X)$ to $s_{\theta'}(X)$. 

Now, restricting to the $T$-fixed points, we get 
\begin{gather*}
0\longrightarrow 
D_\lambda\otimes N
\longrightarrow 
R_\lambda 
\longrightarrow M\longrightarrow 0,\\[5pt]
0\longrightarrow 
D_\mu\otimes N
\longrightarrow 
R_\mu 
\longrightarrow M\longrightarrow 0.  
\end{gather*}

\begin{proof}[Proof of Theorem \ref{thm:main}]
Assume that the graded Frobenius characteristics of $M$ and $N$ are $F_M$ and $F_N$, respectively. 
Then we obtain that 
\begin{align*}
\widetilde{H}_\lambda 
 &=\operatorname{Frob}_{q,t}(R_\lambda)
= \operatorname{Frob}_{q,t}(D_\lambda\otimes N)+\operatorname{Frob}_{q,t}(M)
= T_\lambda\cdot F_N +F_M\\
\widetilde{H}_\mu
& =\operatorname{Frob}_{q,t}(R_\mu)
= \operatorname{Frob}_{q,t}(D_\mu\otimes N)+\operatorname{Frob}_{q,t}(M) 
= T_\mu\cdot F_N +F_M. 
\end{align*}
Note that $S_n$ acts on $\mathcal{O}_{H_n}(1)$ trivially. 
This leads to 
$$\frac{T_\lambda\widetilde{H}_{\mu}-T_\mu\widetilde{H}_\lambda}{T_\lambda-T_\mu} = F_M
\in \sum_{\theta\vdash n}\mathbb{Z}_{\geq 0}[q^{\pm 1},t^{\pm 1}]s_\theta.$$
Since $M$ is a quotient of $R_\lambda$, we have 
$F_M
\in \sum_{\theta\vdash n}\mathbb{Z}_{\geq 0}[q,t]s_\theta$
as desired. 
\end{proof}

\end{document}